\documentclass[12pt]{article}
\usepackage{amsmath,amsfonts,amssymb,amsthm,latexsym}
\theoremstyle{plain}
\numberwithin{equation}{section}
\newtheorem{thm}{Theorem}[section]
\newtheorem{theorem}[thm]{Theorem}
\newtheorem{lemma}[thm]{Lemma}
\newtheorem{corollary}[thm]{Corollary}

\newcommand{\N}{{\mathbb N}}

\begin{document}

\begin{center}
\textbf{\large The general case on the order of appearance of product of consecutive Fibonacci and Lucas numbers}
\end{center}

\begin{center}
Narissara Khaochim and Prapanpong Pongsriiam

Department of Mathematics, Faculty of Science\\
                Silpakorn University\\
                Nakhon Pathom\\
               73000, Thailand
							
				narissara.khaochim@gmail.com 
							
				prapanpong@gmail.com			
\end{center}

\begin{abstract}
Let $F_{n}$ and $L_n$ be the $n$th Fibonacci and Lucas number, respectively. For each positive integer $m$, the order of appearance of $m$ in the Fibonacci sequence, denoted by $z(m)$, is the smallest positive integer $k$ such that $m$ divides $F_k$. Recently, D. Marques has obtained a formula for $z(F_{n}F_{n+1})$, $z(F_{n}F_{n+1}F_{n+2})$, and $z(F_{n}F_{n+1}F_{n+2}F_{n+3})$. In this paper, we extend Marques' result to the case $z(F_{n}F_{n+1}\cdots F_{n+k})$ for every $4\leq k \leq 6$. We also give a formula for $z(L_nL_{n+1}\cdots L_{n+k})$ when $k = 5,6$ which extends the recent result of Marques and Trojovsk\'y. Our method gives a general idea on how to obtain the formulas for $z(F_nF_{n+1}\cdots F_{n+k})$ and $z(L_nL_{n+1}\cdots L_{n+k})$ for every $k\geq 1$. 
\end{abstract}


\section{Introduction}
Throughout this article, we write $(a_1,a_2,\ldots,a_k)$ and $[a_1,a_2,\ldots,a_k]$ for the greatest common divisor and the least common multiple of $a_1, a_2, \ldots, a_k$, respectively.

The Fibonacci sequence $\left(F_n\right)_{n\geq 1}$ is defined by $F_1=F_2=1$ and $F_n=F_{n-1}+F_{n-2}$ for $n\geq 3$. The Lucas sequence $(L_n)_{n\geq 1}$ is given by the same recursive pattern as the Fibonacci sequence but with the initial values $L_1 = 1$ and $L_2 = 3$. For each $m\in \N$, the order of appearance of $m$ in the Fibonacci sequence, denoted by $z(m)$, is the smallest positive integer $k$ such that $m$ divides $F_{k}$. The divisibility property of Fibonacci numbers and the behavior of the order of appearance have been a popular area of research, see \cite{BeandQu, Gr, Ha, Ka, Ko, PP1, PP2, PP3, Ro, Va, Vi, Wa} and references therein for additional details and history. Recently, D. Marques \cite{Ma1,Ma2,Ma3,Ma4,Ma5} has obtained formulas for $z(m)$ for various types of $m$. In particular, Marques \cite{Ma1, Ma4} obtains the formulas for $z(F_{n}F_{n+1})$, $z(F_{n}F_{n+1}F_{n+2})$, $z(F_{n}F_{n+1}F_{n+2}F_{n+3})$, $z(L_nL_{n+1})$, $z(L_nL_{n+1}L_{n+2})$, and $z(L_nL_{n+1}L_{n+2}L_{n+3})$. Marques and Trojovsk\'y \cite{MTr} also compute a formula for $z(L_nL_{n+1}\cdots L_{n+4})$.

In this article, we extend those results to the case $z(F_{n}F_{n+1}\cdots F_{n+k})$ for every $4\leq k \leq 6$ and $z(L_nL_{n+1}\cdots L_{n+k})$ for $k = 5, 6$. Our method is simpler and gives a general idea on how to obtain the formulas of $z(F_{n}F_{n+1}\cdots F_{n+k})$ and  $z(L_nL_{n+1}\cdots L_{n+k})$ for every $k\geq 1$. The general cases of $z(F_nF_{n+1}\cdots F_{n+k})$ and $z(L_nL_{n+1}\cdots L_{n+k})$ are of interest to us since they are connected to the widely study of the least common multiple of consecutive integers, which is initiated by Chebyshev for the first significant attempt to prove the prime number theorem. We refer the reader to the articles by Farhi \cite{Far2}, Farhi and Kane \cite{Far}, Hong, Luo, Qian, and Wang \cite{HLQW}, and Hong and Qian \cite{HQ} and references therein for the investigation of the least common multiple of some finite sequences of integers. 

We arrange the article as follows. In the next section, we give some auxiliary results that are useful for the proof of main theorems. Then we calculate $z(F_nF_{n+1}\cdots F_{n+k})$ in Section \ref{sec3}, and $z(L_nL_{n+1}\cdots L_{n+k})$ in Section \ref{sec4}. Since the method is the same, we show full details in the case $z(F_nF_{n+1}\cdots F_{n+k})$ and give only an outline in the other case.
\section{Auxiliary Results}
In this section, we give some lemmas that will be used in the proof of the main theorems. First we recall the following well-known results \cite{Gr,Ka,Ko,Va} which will be applied throughout this article :
\begin{equation}\label{eq:1}
\text{For} \ n\geq3, m\geq1, \ F_{n}\mid F_{m} \ \text{if and only if} \ n\mid m. 
\end{equation}
\begin{equation}\label{eq:gcdF}
\text{For} \ m,n\geq1, \ (F_{m},F_{n})=F_{(m,n)}. 
\end{equation}
We will need to calculate 2-adic and 3-adic orders of Fibonacci numbers and the next lemma will be useful. 
\begin{lemma} \label{p-adic} (Lengyel \cite{Le})
For each $n\geq1$, let $v_{p}(n)$ be the p-adic order of $n$. Then
\[
v_{2}(F_{n})=\begin{cases}
0, & if \ n\equiv1,2\pmod {3};\\
1, & if \ n\equiv3\pmod {6};\\
v_2(n)+2, & if \ n\equiv0\pmod {6},\\
\end{cases}
\]
$v_{5}(F_n)=v_{5}(n)$, and if p is a prime, $p\not=2$, and $p\not=5$, then
\[
v_{p}(F_{n})=\begin{cases}
v_{p}(n)+v_{p}(F_{z(p)}), & if \ n\equiv0\pmod {z(p)};\\
0, & if \ n\not\equiv0\pmod {z(p)}.\\
\end{cases}
\]
In particular, 
\[
v_{3}(F_{n})=\begin{cases}
v_{3}(n)+1, & if \ n\equiv0\pmod {4};\\
0, & if \ n\not\equiv0\pmod {4}.\\
\end{cases}
\]
\end{lemma}
We will also need to calculate the least common multiple of consecutive integers such as $[n,n+1,n+2,n+3,n+4]$. It is not difficult to compute directly the formula for  $[n,n+1,\ldots,n+k]$ in terms of $n,n+1,\ldots,n+k$ for $1\leq k \leq 6$. But it is more convenient to apply the result of Farhi and Kane \cite{Far} on the recursive relation of the function $g_{k}:\N \rightarrow \N$ given by 
\begin{equation} \label{eq:fh}
g_{k}(n)=\frac{n(n+1)\cdots(n+k)}{[n,n+1,\ldots,n+k]}.
\end{equation}
\begin{lemma} \label{Farhi} (Farhi and Kane\cite{Far})
For each $k\in \N \cup \left\{0\right\}$, let $g_k$ be the function defined by (\ref{eq:fh}). Then $g_0(n)=g_1(n)=1$ for every $n \in \N$ and $g_k$ satisfies the recursive relation
$$
g_{k}(n)=(k!,(n+k)g_{k-1}(n)) \ \text{for all $k,n\in \N$.}
$$
\end{lemma}
Let $a, b, c$ be positive integers. Recall the basic results in elementary number theory that if $(a,b)=1$, then $(c,ab)=(c,a)(c,b)$, and $(a,bc)=(a,c)$. In addition, $\left((a,b),c\right)=(a,b,c)$, $(a,b) = (b,a)$, $(ca,cb)=c(a,b)$, and if $a\equiv b \pmod c$, then $(a,c)=(b,c)$. Combining these and Lemma \ref{Farhi}, we obtain the following result.
\begin{lemma} \label{lcmN} 
For each $n\in \N$, the following holds.
\begin{align*}
[n,n+1]&=n(n+1), \\ 
[n,n+1,n+2]&=\frac{n(n+1)(n+2)}{(2,n)}, \\
[n,n+1,n+2,n+3]&=\frac{n(n+1)(n+2)(n+3)}{2(3,n)}, \\
[n,n+1,n+2,n+3,n+4]&=\frac{n(n+1)(n+2)(n+3)(n+4)}{2(4,n)(3,n(n+1))}, \\
[n,n+1,n+2,n+3,n+4,n+5]&=\frac{n(n+1)(n+2)(n+3)(n+4)(n+5)}{6(5,n)(4,n(n+1))}, \\
[n,n+1,n+2,n+3,n+4,n+5,n+6]&=\frac{n(n+1)(n+2)(n+3)(n+4)(n+5)(n+6)}{12(3,n)(5,n(n+1))\left(4,(n+2)\left(2,\frac{n(n+1)}{2}\right)\right)}.
\end{align*}
\end{lemma}
\begin{proof}
By the definition of the function $g_{k}(n)$, we obtain that $[n,n+1,\ldots ,n+k]=\frac{n(n+1)\cdots(n+k)}{g_{k}(n)}$. So we only need to find $g_{k}(n)$ for $k=1,2,3,4,5,6$. Since each case is similar, we will only give the proof in the cases $k=5,6$ assuming that the cases $k=1,2,3,4$ are already obtained.\\
\textbf{Case 1 : $k=5$}. Assuming that the case $k=4$ is proved, we have $g_{4}(n)=2(4,n)(3,n(n+1))$ and we obtain by Lemma \ref{Farhi} that
\begin{align*}
g_{5}(n)&=(5!,(n+5)g_{4}(n)) \\
&=(5!,2(n+5)(4,n)(3,n(n+1))) \\
&=2(5\cdot4\cdot3,(n+5)(4,n)(3,n(n+1))) \\
&=2(5,n+5)(4,(n+5)(4,n))(3,(n+5)(3,n(n+1))) \\
&=2(5,n)(4,(n+1)(4,n))(3,3(n+5),n(n+1)(n+5)) \\
&=2(5,n)(4,4(n+1),n(n+1))(3,n(n+1)(n+5)) \\
&=2(5,n)(4,n(n+1))3 \\
&=6(5,n)(4,n(n+1)). 
\end{align*}
\textbf{Case 2 : $k=6$}. \ We have
\begin{align*}
g_{6}(n)&=(6!,(n+6)g_{5}(n)) \\
&=(6!,6(n+6)(5,n)(4,n(n+1))) \\
&=6(8\cdot5\cdot3,(n+6)(5,n)(4,n(n+1))) \\
&=6(8,(n+6)(4,n(n+1)))(5,(n+6)(5,n))(3,n+6) \\
&=6(8,(n+6)(4,n(n+1)))(5,(n+1)(5,n))(3,n) \\
&=6(8,(n+6)(4,n(n+1)))(5,5(n+1),n(n+1))(3,n) \\
&=12\left(4,(n+6)\left(2,\frac{n(n+1)}{2}\right)\right)(5,n(n+1))(3,n) \\
&=12\left(4,(n+2)\left(2,\frac{n(n+1)}{2}\right)\right)(5,n(n+1))(3,n).
\end{align*} 
\end{proof}
Next we calculate the least common multiple of consecutive Fibonacci numbers.
\begin{lemma} \label{lcmF} 
For each $n\in \N$, the following holds.\\
\begin{itemize}
\item[(i)] $[F_n,F_{n+1}]=F_nF_{n+1}$. \\ 
\item[(ii)] $[F_n,F_{n+1},F_{n+2}]=F_nF_{n+1}F_{n+2}$. \\
\item[(iii)] $[F_n,F_{n+1},F_{n+2},F_{n+3}]=\frac{F_nF_{n+1}F_{n+2}F_{n+3}}{F_{(n,3)}}$. \\
\item[(iv)] $
[F_n,F_{n+1},F_{n+2},F_{n+3},F_{n+4}]=\begin{cases}
\frac{F_nF_{n+1}F_{n+2}F_{n+3}F_{n+4}}{F_{(n,4)}}, & \text{if $n\equiv1\pmod {3}$};\\
\frac{F_nF_{n+1}F_{n+2}F_{n+3}F_{n+4}}{2F_{(n,4)}}, & \text{if $n\equiv0,2\pmod {3}$}.\\
\end{cases}$
\item[(v)] $
[F_n,F_{n+1},F_{n+2},\ldots,F_{n+5}]=\begin{cases}
\frac{F_nF_{n+1}F_{n+2}F_{n+3}F_{n+4}F_{n+5}}{2F_{(n,5)}}, & \text{if $n\equiv1,2\pmod {4}$};\\
\frac{F_nF_{n+1}F_{n+2}F_{n+3}F_{n+4}F_{n+5}}{6F_{(n,5)}}, & \text{if $n\equiv0,3\pmod {4}$}.\\
\end{cases}$
\item[(vi)] 
$[F_n,F_{n+1},F_{n+2},\ldots,F_{n+6}]
= \begin{cases}
\frac{F_nF_{n+1}F_{n+2}F_{n+3}F_{n+4}F_{n+5}F_{n+6}}{2F_{(n(n+1),5)}F_{(n,6)}}, & \text{if $n\equiv1\pmod {4}$};\\
\frac{F_nF_{n+1}F_{n+2}F_{n+3}F_{n+4}F_{n+5}F_{n+6}}{6F_{(n(n+1),5)}F_{(n,6)}}, & \text{if $n\equiv0,2,3\pmod {4}$}.
\end{cases}
$
\end{itemize}
\end{lemma}
\begin{proof}
By (\ref{eq:gcdF}), it is easy to check that $F_n,F_{n+1},F_{n+2}$ are pairwise relatively prime. So (i) and (ii) follow immediately. Since (iii), (iv), (v), and (vi) follow from the same idea, we will only show the proof for (iii), (v), and (vi). 

Recall that $[a_1,a_2,\ldots,a_k]=\left[[a_1,a_2,\ldots,a_{k-1}],a_k\right]$ and $[a,b]=\frac{ab}{(a,b)}$. For convenience, we let $P_k=F_nF_{n+1}\cdots F_{n+k}$. Then (iii) follows from (ii) by 
\begin{align*} 
[F_n,F_{n+1},F_{n+2},F_{n+3}]&=\left[[F_n,F_{n+1},F_{n+2}],F_{n+3}\right]\\
&=\frac{[F_n,F_{n+1},F_{n+2}]F_{n+3}}{\left([F_n,F_{n+1},F_{n+2}],F_{n+3}\right)}=\frac{F_nF_{n+1}F_{n+2}F_{n+3}}{\left(F_nF_{n+1}F_{n+2},F_{n+3}\right)}\\
&=\frac{P_3}{(F_n,F_{n+3})(F_{n+1},F_{n+3})(F_{n+2},F_{n+3})}\\
&=\frac{P_3}{F_{(n,n+3)}}=\frac{P_3}{F_{(n,3)}}.
\end{align*}
Assuming (iv), we can obtain (v) in the following similar way. Since $F_{n+3}$, $F_{n+4}$, $F_{n+5}$ are pairwise relatively prime, we see that
\begin{align}
(P_4,F_{n+5})&=(F_n,F_{n+5})(F_{n+1},F_{n+5})(F_{n+2},F_{n+5})  \notag \\
&=F_{(n,n+5)}F_{(n+1,n+5)}F_{(n+2,n+5)} \notag \\
&=F_{(n,5)}F_{(n+1,4)}F_{(n+2,3)}. \label{eq:gcd5_1}
\end{align}
\textbf{Case 1: $n\equiv 1 \pmod 3$}. Then
\begin{align*}  
[F_n,F_{n+1},F_{n+2},F_{n+3},F_{n+4},F_{n+5}]&=[[F_n,F_{n+1},F_{n+2},F_{n+3},F_{n+4}],F_{n+5}]\\
&=\left[\frac{F_nF_{n+1}F_{n+2}F_{n+3}F_{n+4}}{F_{(n,4)}},F_{n+5}\right]\\
&=\frac{F_nF_{n+1}F_{n+2}F_{n+3}F_{n+4}F_{n+5}}{F_{(n,4)}\left(\frac{F_nF_{n+1}F_{n+2}F_{n+3}F_{n+4}}{F_{(n,4)}},F_{n+5}\right)}\\
&=\frac{F_nF_{n+1}F_{n+2}F_{n+3}F_{n+4}F_{n+5}}{(F_nF_{n+1}F_{n+2}F_{n+3}F_{n+4},F_{(n,4)}F_{n+5})}\\
&=\frac{P_5}{(P_4,F_{(n,4)}F_{n+5})}.
\end{align*}
Since $\left(F_{(n,4)},F_{n+5}\right)=F_{((n,4),n+5)}=F_{(n,(4,n+5))}=F_{(n,4,n+1)}=1$ and $n\equiv 1 \pmod 3 $, we obtain by (\ref{eq:gcd5_1}) that 
\begin{equation} \label{eq:gcd5_2}
(P_4,F_{(n,4)}F_{n+5})=2(P_4,F_{(n,4)})F_{(n,5)}F_{(n+1,4)}.
\end{equation}
It is easy to check that if $n\equiv 1,2 \pmod 4$, then the right hand side of (\ref{eq:gcd5_2}) is equal to $2F_{(n,5)}$, and if $n\equiv 0,3 \pmod 4$, then it is equal to $6F_{(n,5)}$. \\
\textbf{Case 2: $n\equiv 0,2 \pmod {3}$}. Similar to Case 1, we have 
\begin{align*}
[F_n,F_{n+1},F_{n+2},F_{n+3},F_{n+4},F_{n+5}]&=\frac{P_5}{(P_4,2F_{(n,4)}F_{n+5})}.
\end{align*}
It is easy to check using (\ref{eq:gcdF}) that $2=F_3$ is relatively prime to $F_{(n,4)}$ and $F_{n+5}$, and that $(F_{(n,4)},F_{n+5})=F_{((n,4),n+5)}=1$. This and (\ref{eq:gcd5_1}) implies that 
$$
(P_4,2F_{(n,4)}F_{n+5}) = 2(P_4,F_{(n,4)})F_{(n,5)}F_{(n+1,4)},
$$
 which is the same as (\ref{eq:gcd5_2}). So if $n\equiv 1,2 \pmod 4$, then it is equal to $2F_{(n,5)}$, and if $n\equiv 0,3 \pmod 4$, then it is equal to $6F_{(n,5)}$. This proves (v). Next we give a proof of (vi).\\
\textbf{Case 1: $n\equiv 1,2 \pmod {4}$}. Similar to the proof of (v), we have
$$
[F_n,F_{n+1},F_{n+2},F_{n+3},F_{n+4},F_{n+5},F_{n+6}]=\frac{P_6}{(P_5,2F_{(n,5)}F_{n+6})}.
$$
It is easy to see that $F_{(n,5)}$ is relatively prime to $2$. This implies that $(F_{(n,5)},2F_{n+6})=(F_{(n,5)},F_{n+6})=F_{\left((n,5),n+6\right)}=1$. So
$$
(P_5,2F_{(n,5)}F_{n+6})=(P_5,F_{(n,5)})(P_5,2F_{n+6})=(F_nF_{n+5},F_{(n,5)})(P_5,2F_{n+6}).
$$
We see that if $5\mid n$, then $(F_nF_{n+5},F_{(n,5)})=5$, and if $5\nmid n$, then $(F_nF_{n+5},F_{(n,5)})=1$. This implies that $(F_nF_{n+5},F_{(n,5)})=F_{(n,5)}$. Thus the above equation becomes
\begin{equation} \label{eq:An}
(P_5,2F_{(n,5)}F_{n+6})=F_{(n,5)}(P_5,2F_{n+6}).
\end{equation}
Consider $(2,F_{n+6})=(F_3,F_{n+6})=F_{(3,n+6)}=F_{(3,n)}$.\\
\textbf{Case 1.1: $3\nmid n$}. Then $(2,F_{n+6})=1$, and $F_{n+6}$ is relatively prime to $F_{n+5}, F_{n+4}$, and $F_{n+3}$. So (\ref{eq:An}) becomes
\begin{align} 
(P_5,2F_{(n,5)}F_{n+6})&=2F_{(n,5)}(P_5,F_{n+6}) \notag \\
&=2F_{(n,5)}(F_nF_{n+1}F_{n+2},F_{n+6}) \notag \\
&=2F_{(n,5)}(F_n,F_{n+6})(F_{n+1},F_{n+6})(F_{n+2},F_{n+6}) \notag \\
&=2F_{(n,5)}F_{(n,6)}F_{(n+1,5)}F_{(n+2,4)} \notag \\
&=2F_{(n(n+1),5)}F_{(n,6)}F_{(n+2,4)}. \label{eq:10} 
\end{align} 
\textbf{Case 1.2: $3\mid n$}. Then $2$ and $F_{n+6}$ are relatively prime to $F_{n+4}$ and $F_{n+5}$. 
In addition, $(F_nF_{n+1}F_{n+2},F_{n+3})=(F_{n},F_{n+3})=F_{(n,3)}=2$. So $\left(\frac{F_nF_{n+1}F_{n+2}}{2},\frac{F_{n+3}}{2}\right)=1$. Therefore
\begin{align*}
(P_5,2F_{n+6})&=(F_nF_{n+1}F_{n+2}F_{n+3},2F_{n+6})\\
&=4\left(\frac{F_nF_{n+1}F_{n+2}F_{n+3}}{4},\frac{F_{n+6}}{2}\right)=4\left(\frac{F_nF_{n+1}F_{n+2}}{2}\frac{F_{n+3}}{2},\frac{F_{n+6}}{2}\right)\\
&=4\left(\frac{F_nF_{n+1}F_{n+2}}{2},\frac{F_{n+6}}{2}\right)\left(\frac{F_{n+3}}{2},\frac{F_{n+6}}{2}\right)\\
&=(F_nF_{n+1}F_{n+2},F_{n+6})(F_{n+3},F_{n+6})\\
&=(F_n,F_{n+6})(F_{n+1},F_{n+6})(F_{n+2},F_{n+6})(F_{n+3},F_{n+6})\\
&=F_{(n,6)}F_{(n+1,5)}F_{(n+2,4)}F_{(n+3,3)}=2F_{(n,6)}F_{(n+1,5)}F_{(n+2,4)}.
\end{align*}
Thus (\ref{eq:An}) becomes
$$
(P_5,2F_{(n,5)}F_{n+6})=2F_{(n,5)}F_{(n,6)}F_{(n+1,5)}F_{(n+2,4)}=2F_{(n(n+1),5)}F_{(n+2,4)}F_{(n,6)},
$$
which is the same as (\ref{eq:10}). We conclude that Case 1.1 and 1.2 lead to the same formula for $[F_n,F_{n+1},F_{n+2},F_{n+3},F_{n+4},F_{n+5},F_{n+6}]$. Observe that if $n\equiv 1\pmod 4$, then $F_{(n+2,4)} = 1$, and if $n\equiv 2\pmod 4$, then $F_{(n+2,4)} = 3$. This leads to the desired formula in (vi).\\
\textbf{Case 2: $n\equiv 0,3 \pmod {4}$}. Similar to the proof of $(v)$, we have 
$$
[F_n,F_{n+1},F_{n+2},F_{n+3},F_{n+4},F_{n+5},F_{n+6}]=\frac{P_6}{(P_5,6F_{(n,5)}F_{n+6})}.
$$
It is easy to see that $F_{(n,5)}$ is relatively prime to $2$ and $3$. So $(F_{(n,5)},6F_{n+6})=(F_{(n,5)},F_{n+6})=F_{((n,5),n+6)}=1$. Thus 
\begin{equation} \label{eq:Bn}
(P_5,6F_{(n,5)}F_{n+6})=(P_5,F_{(n,5)})(P_5,6F_{n+6})=F_{(n,5)}(P_5,6F_{n+6}).
\end{equation}
\textbf{Case 2.1: $3\nmid n$}. Then $(6,F_{n+6})=1$ and $(F_{n+3}F_{n+4}F_{n+5},F_{n+6})=1$. So $(P_5,6F_{n+6})=6(P_5,F_{n+6})=6(F_nF_{n+1}F_{n+2},F_{n+6})=6(F_n,F_{n+6})(F_{n+1},F_{n+6})(F_{n+2},F_{n+6})=6F_{(n,6)}F_{(n+1,5)}$. So we obtain by (\ref{eq:Bn}) that 
\begin{equation} \label{eq:12}
(P_5,6F_{(n,5)}F_{n+6})=6F_{(n,5)}F_{(n,6)}F_{(n+1,5)}=6F_{(n,6)}F_{(n(n+1),5)}.
\end{equation}
\textbf{Case 2.2: $3\mid n$}. Then $(F_{n+5},6F_{n+6})=(F_{n+5},6)=(F_4,F_{n+5})(F_3,F_{n+5})=F_{(4,n+1)}$. We obtain similarly that $(F_{n+4},6F_{n+6})=F_{(4,n)}$ and $(F_{n+3},6F_{n+6})$ $=(F_{n+3},3)(F_{n+3},2F_{n+6})$ $=(F_{n+3},2F_{n+6}) = (F_{n+3},4)$, where the last equality is obtained from the fact that $(F_{n+3},F_{n+6})=2$. So
\begin{equation} \label{eq:Cn}
(F_{n+3}F_{n+4}F_{n+5},6F_{n+6})=F_{(4,n+1)}F_{(4,n)}(F_{n+3},4).
\end{equation}
From this we obtain by Lemma \ref{p-adic} that
\[
(F_{n+3}F_{n+4}F_{n+5},6F_{n+6})=
\begin{cases}
6, & \text{if $n\equiv 0 \pmod {12}$};\\
12, & \text{if $n\equiv 3 \pmod {12}$}.
\end{cases}
\] 
\textbf{Case 2.2.1: $n\equiv 0 \pmod {12}$}. Then $\left(\frac{F_{n+3}F_{n+4}F_{n+5}}{6},F_{n+6}\right)=1$. So
\begin{align*}
(P_5,6F_{n+6})&=6\left(F_nF_{n+1}F_{n+2}\frac{F_{n+3}F_{n+4}F_{n+5}}{6},F_{n+6}\right)\\
&=6(F_nF_{n+1}F_{n+2},F_{n+6})\\
&=6(F_n,F_{n+6})(F_{n+1},F_{n+6})(F_{n+2},F_{n+6})\\
&=6F_{(n,6)}F_{(n+1,5)}.
\end{align*}
Thus we obtain by (\ref{eq:Bn}) that 
\begin{equation} \label{eq:13}
(P_5,6F_{(n,5)}F_{n+6})=6F_{(n,6)}F_{(n+1,5)}F_{(n,5)}=6F_{(n,6)}F_{(n(n+1),5)},
\end{equation}
which is the same as (\ref{eq:12}).\\
\textbf{Case 2.2.2: $n\equiv 3 \pmod {12}$}. Then $\left(\frac{F_{n+3}F_{n+4}F_{n+5}}{12},\frac{F_{n+6}}{2}\right)=1$. So
\begin{align*}
(P_5,6F_{n+6})&=12\left(F_nF_{n+1}F_{n+2}\frac{F_{n+3}F_{n+4}F_{n+5}}{12},\frac{F_{n+6}}{2}\right)\\
&=12\left(F_nF_{n+1}F_{n+2},\frac{F_{n+6}}{2}\right)\\
&=12\left(F_n,\frac{F_{n+6}}{2}\right)\left(F_{n+1},\frac{F_{n+6}}{2}\right)\left(F_{n+2},\frac{F_{n+6}}{2}\right).
\end{align*}
Consider $(F_{n+2},F_{n+6})=F_{(n+2,4)}=1$, $(F_{n+1},F_{n+6})=F_{(n+1,5)}$, $(F_{n},F_{n+6})=F_{(n,6)}=F_{(3,6)}=2$, and $v_2(F_n)=v_2(F_{n+6})=1$. Therefore $(P_5,6F_{n+6})=12F_{(n+1,5)}$, and thus $(P_5,6F_{(n,5)}F_{n+6})=12F_{(n,5)}F_{(n+1,5)}=6F_{(n,6)}F_{(n(n+1),5)}$, which is the same as (\ref{eq:13}) and (\ref{eq:12}). So Case 2.1 and 2.2 lead to the same formula for $[F_n,F_{n+1},F_{n+2},F_{n+3}F_{n+4},F_{n+5}F_{n+6}]$. This completes the proof of (vi).
\end{proof}
\section{Main Results}\label{sec3}
As mentioned in the introduction, our method of proof gives a general idea on how to obtain $z(F_{n}F_{n+1}\cdots F_{n+k})$ for every $k\geq 1$. In fact, the next theorem describes the general strategy in obtaining the formula for $z(F_{n}F_{n+1}\cdots F_{n+k})$. 
\begin{theorem} \label{idea}
Let $n\geq 3$, $k\geq 1$, $a=[n,n+1,\ldots,n+k]$, $b = F_nF_{n+1}\cdots F_{n+k}$ and $f_k(n)=\frac{F_nF_{n+1}F_{n+2}\cdots F_{n+k}}{[F_n,F_{n+1},F_{n+2},\ldots ,F_{n+k}]}$. Then the following holds.
\begin{itemize}
\item[(i)] $b\mid f_k(n)F_{aj}$ for every $j\geq 1$.
\item[(ii)] $z(b)=aj$ where $j$ is the smallest positive integer such that $b\mid F_{aj}$. In fact, $j$ is the smallest positive integer such that $v_p(b) \leq v_p(F_{aj})$ for every prime $p$ dividing $f_k(n)$.
\end{itemize}
\end{theorem}
\begin{proof}
Since $n+i\mid a$ for all $0\leq i\leq k$, we obtain by (\ref{eq:1}) that $F_{n+i}\mid F_a$ for all $0\leq i\leq k$. So $[F_n,F_{n+1},\ldots,F_{n+k}]\mid F_a$. By the definition of $f_k(n)$, we see that $b\mid f_k(n)F_a$. Since $F_a\mid F_{aj}$,
$$
b\mid f_k(n)F_{aj} \ \ \text{for every $j\geq 1$}. 
$$
This proves (i). Next let $z(b)=\ell$. Then $b\mid F_\ell$. Therefore $F_{n+i}\mid F_\ell$ for all $0\leq i\leq k$. Since $n\geq3$, we obtain by (\ref{eq:1}) that $n+i\mid \ell$ for all $0\leq i\leq k$, which implies that $a\mid \ell$. Thus $\ell=aj$ for some $j\in \N$. By the definition of $z(b)$, we see that $j$ is the smallest positive integer such that
\begin{equation}\label{eq:equa1} 
b\mid F_{aj}.
\end{equation} 
Note that (\ref{eq:equa1}) is equivalent to $v_{p}(b)\leq v_{p}(F_{aj})$ for every prime $p$. But by (i), if $p$ is a prime and $p\nmid f_k(n)$, then 
$$
v_{p}(b)\leq v_{p}(f_k(n)F_{aj})=v_{p}(F_{aj}). 
$$
Therefore (\ref{eq:equa1}) is equivalent to 
\begin{equation} \label{eq:equa4}
v_{p}(b)\leq v_{p}(F_{aj}) \ \text{for every prime $p$ dividing $f_k(n)$}. 
\end{equation}
Hence $z(b) = \ell = aj$ and $j$ is the smallest positive integer satisfying (\ref{eq:equa4}). This proves (ii).
\end{proof}
\begin{thm}\label{order4}
Let $n\geq1$, $a=[n,n+1,n+2,n+3,n+4]$, and $b = F_{n}F_{n+1}F_{n+2}F_{n+3}F_{n+4}$. Then
\[
z(b) =
\begin{cases}
a, & \text{if $n\equiv1,2,3,4,5,6,7,10 \pmod {12}$, or $n\equiv8,60 \pmod {72}$};\\
2a, & \text{if $n\equiv9,11\pmod {12}$, or $n\equiv24,44 \pmod {72}$};\\
3a, & \text{if $n\equiv12,32,36,56 \pmod {72}$}; \\
6a, & \text{if $n\equiv0,20,48,68 \pmod {72}$}.
\end{cases}
\] 
\end{thm}
\begin{proof} 
It is easy to check that the result holds for $n=1, 2$. So assume that $n\geq 3$. \\
\textbf{Case 1: $n\equiv 1 \pmod {3}$}. Then by Lemma \ref{lcmF} and Theorem \ref{idea}, we have $b\mid F_{(n,4)}F_{aj}$ for every $j\geq 1$ and we would like to find the smallest $j$ such that $b\mid F_{aj}$. If $n\equiv 1,2,3 \pmod {4}$, then $F_{(n,4)}=1$, so we can choose $j=1$ and obtain $z(b)=a$. So assume that $n\equiv 0 \pmod {4}$. Then $F_{(n,4)}=3$ and by Theorem \ref{idea} we only need to consider $v_3(b)$ and $v_3(F_{aj})$. Since $n\equiv 1 \pmod {3}$ and $n\equiv 0 \pmod {4}$, we obtain by Lemma \ref{p-adic} that $v_3(b)=v_3(F_{n})+v_3(F_{n+4})=v_3(n)+v_3(n+4)+2=2$. Since $4\mid n$ and $n\mid aj$, $4\mid aj$. So we obtain by Lemmas \ref{p-adic} and \ref{lcmN} that for every $j\geq 1$,
\begin{align*}
v_3(F_{aj})&=v_3(a)+v_3(j)+1=v_3\left(\frac{n(n+1)(n+2)(n+3)(n+4)}{8}\right)+v_3(j)+1\\
&=v_3(n+2)+v_3(j)+1\geq 2+v_3(j)\geq 2=v_3(b).
\end{align*}
Thus we can choose $j=1$ and obtain $z(b)=a$. This shows $z(b) = a$ whenever $n\equiv 1\pmod 3$. We remark that the idea that will be used in the following case is still the same as that in the previous case. So our argument will be shorter.\\
\textbf{Case 2: $n\equiv 2 \pmod {3}$}. Then by Lemma \ref{lcmF} and Theorem \ref{idea}, we have $b\mid 2F_{(n,4)}F_{aj}$ for every $j\geq 1$ and our problem is reduced to finding the smallest positive integer $j$ such that $v_p(b)\leq v_p(F_{aj})$ for every prime $p$ dividing $2F_{(n,4)}$. Let $j\geq 1$. Since $3\mid n+1$ and $n+1\mid a$, we see that $3\mid aj$. Similarly $2\mid aj$. Therefore $6\mid aj$. By Lemma \ref{p-adic}, $v_2(F_{aj}) = v_2(aj)+2$. In addition, $v_2(b) = v_2(F_{n+1})+v_2(F_{n+4})$.\\
\textbf{Case 2.1: $n\equiv 1 \pmod {4}$}. Then by Lemmas \ref{p-adic} and \ref{lcmN}, we obtain 
\begin{align*}
v_2(F_{aj})&=v_2(a)+v_2(j)+2=v_2(n+1)+v_2(n+3)-v_2(2)+v_2(j)+2\\
&=v_2(n+3)+v_2(j)+2\geq 4 = v_2(n+1)+3 = v_2(F_{n+1})+v_2(F_{n+4})=v_2(b).
\end{align*}
So in this case, we can choose $j = 1$ and obtain $z(b)=a$.\\
\textbf{Case 2.2: $n\equiv 2 \pmod {4}$}. Similar to Case 2.1, we see that 
\begin{align*}
v_2(F_{aj})&=v_2(n)+v_2(n+2)+v_2(n+4)-v_2(4)+v_2(j)+2\\
&=v_2(n+2)+v_2(j)+2\geq 4=v_2(b), \ \text{and} \ z(b)=a.
\end{align*}
\textbf{Case 2.3: $n\equiv 3 \pmod {4}$}. Then $v_2(b)=v_2(n+1)+3$, and $v_2(F_{aj})=v_2(n+1)+v_2(j)+2$. So $v_2(F_{aj})\geq v_2(b)$ if and only if $v_2(j)\geq 1$. So we choose $j=2$ and obtain $z(b)=2a$.\\
\textbf{Case 2.4: $n\equiv 0 \pmod {4}$}. 
Then $2F_{(n,4)}=6$ and we need to consider 2-adic and 3-adic orders of $b$ and $F_{aj}$. By Lemmas \ref{p-adic} and \ref{lcmN}, we obtain similarly to the other cases that
\begin{align*}
v_2(b) &= v_2(n+4)+3, \quad\quad v_2(F_{aj})=v_2(n)+v_2(n+4)+v_2(j),\\
v_3(b) &= v_3(F_n)+v_3(F_{n+4})=v_3(n)+v_3(n+4)+2=v_3(n+4)+2, \,\text{and}\\
v_3(F_{aj}) &= v_3(aj)+1 = v_3(n+1)+v_3(n+4)+v_3(j).
\end{align*}
So we need to find the smallest $j\geq 1$ such that
$$
v_2(n)+v_2(j)\geq 3 \ \text{and} \ v_3(n+1)+v_3(j)\geq 2.
$$
Note that $n \equiv 0,4 \pmod 8$ and $n+1 \equiv 0,3,6 \pmod 9$.
\begin{itemize}
\item[(i)] If $n \equiv 0 \pmod 8$ and $n+1 \equiv 0 \pmod 9$, then $v_2(j)=v_3(j)=0$, so $j=1$ and $z(b)=a=\frac{72a}{(8,n)(9,n+1)}$.
\item[(ii)] If $n \equiv 0 \pmod 8$ and $n+1 \equiv 3,6 \pmod 9$, then $v_2(j)=0$ and $v_3(j)=1$, so $j=3$ and $z(b)=3a=\frac{72a}{(8,n)(9,n+1)}$.
\item[(iii)] If $n \equiv 4 \pmod 8$ and $n+1 \equiv 0 \pmod 9$, then $v_2(j)=1$ and $v_3(j)=0$, so $j=2$ and $z(b)=2a=\frac{72a}{(8,n)(9,n+1)}$.
\item[(iv)] If $n \equiv 4 \pmod 8$ and $n+1 \equiv 3,6 \pmod 9$, then $v_2(j)=v_3(j)=1$, so $j=6$ and $z(b)=6a=\frac{72a}{(8,n)(9,n+1)}$.
\end{itemize}
\textbf{Case 3: $n\equiv 0 \pmod {3}$}. Similar to Case 2, $b\mid 2F_{(n,4)}F_{aj}$ for every $j\geq 1$ and we need to find the smallest $j$ such that $v_p(b)\leq v_p(F_{aj})$ for every prime $p$ dividing $2F_{(n,4)}$.\\
\textbf{Case 3.1: $n\equiv 1 \pmod {4}$}. Then $2F_{(n,4)}=2$, $v_2(b)=v_2(n+3)+3$, and $v_2(F_{aj})=v_2(n+3)+v_2(j)+2$. So we need $j=2$ and therefore $z(b)=2a$.\\
\textbf{Case 3.2: $n\equiv 2 \pmod {4}$}. Then $2F_{(n,4)}=2$, $v_2(b)=4$, and $v_2(F_{aj})=v_2(n+2)+v_2(j)+2\geq 4=v_2(b)$. So $j=1$ and $z(b)=a$.\\
\textbf{Case 3.3: $n\equiv 3 \pmod {4}$}. Then $2F_{(n,4)}=2$, $v_2(b)=4$, and $v_2(F_{aj})=v_2(n+1)+v_2(j)+2\geq 4=v_2(b)$. So $j=1$ and $z(b)=a$.\\
\textbf{Case 3.4: $n\equiv 0 \pmod {4}$}. Then $2F_{(n,4)}=6$. So we need to consider 2-adic and 3-adic orders of $b$ and $F_{aj}$. By Lemmas \ref{p-adic} and \ref{lcmN}, we obtain that
\begin{align*}
v_2(b) &= v_2(n)+3, \quad\quad v_2(F_{aj}) = v_2(n)+v_2(n+4)+v_2(j),\\
v_3(b) &= v_3(F_n)+v_3(F_{n+4}) = v_3(n)+v_3(n+4)+2 = v_3(n)+2,\;\text{ and}\\
v_3(F_{aj}) &= v_3(aj)+1 = v_3(n)+v_3(n+3)+v_3(j).
\end{align*} 
 So we need to find the smallest $j\geq 1$ such that
$$
v_2(n+4)+v_2(j)\geq 3 \ \text{and} \ v_3(n+3)+v_3(j)\geq 2.
$$
Note that $n+4 \equiv 0,4 \pmod 8$ and $n+3 \equiv 0,3,6 \pmod 9$.
\begin{itemize}
\item[(i)] If $n+4 \equiv 0 \pmod 8$ and $n+3 \equiv 0 \pmod 9$, then $v_2(j)=v_3(j)=0$, so $j=1$ and $z(b)=a=\frac{72a}{(8,n+4)(9,n+3)}$.
\item[(ii)] If $n+4 \equiv 0 \pmod 8$ and $n+3 \equiv 3,6 \pmod 9$, then $v_2(j)=0$ and $v_3(j)=1$, so $j=3$ and $z(b)=3a=\frac{72a}{(8,n+4)(9,n+3)}$.
\item[(iii)] If $n+4 \equiv 4 \pmod 8$ and $n+3 \equiv 0 \pmod 9$, then $v_2(j)=1$ and $v_3(j)=0$, so $j=2$ and $z(b)=2a=\frac{72a}{(8,n+4)(9,n+3)}$.
\item[(iv)] If $n+4 \equiv 4 \pmod 8$ and $n+3 \equiv 3,6 \pmod 9$, then $v_2(j)=v_3(j)=1$, so $j=6$ and $z(b)=6a=\frac{72a}{(8,n+4)(9,n+3)}$.
\end{itemize}
This completes the proof.
\end{proof}
We can state Theorem \ref{order4} in another form as follows.
\begin{corollary} 
Let $n\geq 1$, $a=[n,n+1,n+2,n+3,n+4]$, and $b=F_{n}F_{n+1}F_{n+2}F_{n+3}F_{n+4}$. Then
\[
z(b) =
\begin{cases}
a, & \text{if $n\equiv1 \pmod {3}$ or $n\equiv2,3,5,6 \pmod {12}$};\\
2a, & \text{if $n\equiv9,11\pmod {12}$};\\
\frac{72a}{(8,n)(9,n+1)}, & \text{if $n\equiv8 \pmod {12}$}; \\
\frac{72a}{(8,n+4)(9,n+3)}, & \text{if $n\equiv0 \pmod {12}$}.
\end{cases}
\] 
\end{corollary}
\begin{proof}
This can be obtained from the proof of Theorem \ref{order4}, or by comparing the result with Theorem \ref{order4}.
\end{proof}
\begin{corollary}
Let $n\geq 1$ and $b=F_{n}F_{n+1}F_{n+2}F_{n+3}F_{n+4}$. Then
\[
z(b) =
\begin{cases}
\frac{n(n+1)(n+2)(n+3)(n+4)}{2}, & \text{if $n\equiv1,7 \pmod {12}$};\\
\frac{n(n+1)(n+2)(n+3)(n+4)}{3}, & \text{if $n\equiv9,11\pmod {12}$};\\
\frac{n(n+1)(n+2)(n+3)(n+4)}{4}, & \text{if $n\equiv10 \pmod {12}$ or $n\equiv0,20,48,68 \pmod {72}$}; \\
\frac{n(n+1)(n+2)(n+3)(n+4)}{6}, & \text{if $n\equiv3,5 \pmod {12}$};\\
\frac{n(n+1)(n+2)(n+3)(n+4)}{8}, & \text{if $n\equiv4\pmod {12}$ or $n\equiv12,32,36,56 \pmod {72}$};\\
\frac{n(n+1)(n+2)(n+3)(n+4)}{12}, & \text{if $n\equiv2,6 \pmod {12}$ or $n\equiv24,44 \pmod {72}$}; \\
\frac{n(n+1)(n+2)(n+3)(n+4)}{24}, & \text{if $n\equiv8,60 \pmod {72}$}.
\end{cases}
\] 
\end{corollary}
\begin{proof}
This follows from Theorem \ref{order4} and Lemma \ref{lcmN}.
\end{proof}
\begin{thm}\label{order5}
Let $n\geq1$, $a=[n,n+1,\ldots,n+5]$, $b = F_nF_{n+1}\cdots F_{n+5}$, and $c = (5,n)$. Then
\[
z(b) =\begin{cases}
ac, & \text{if $n\equiv1,2,3,4,5,6 \pmod {12}$, or} \\
& \text{$n\equiv7,8,59,60 \pmod {72}$};\\
2ac, & \text{if $n\equiv9,10\pmod {12}$, or $n\equiv23,24,43,44 \pmod {72}$};\\
3ac, & \text{if $n\equiv11,12,31,32,35,36,55,56 \pmod {72}$}; \\
6ac, & \text{if $n\equiv0,19,20,47,48,67,68,71 \pmod {72}$}.
\end{cases}
\]
\end{thm}
\begin{proof}
The proof of this theorem is similar to that of Theorem \ref{order4}. So we will be brief here. It is easy to check that the result holds for $n=1,2$. So assume that $n\geq 3$. By Lemma \ref{lcmF} and Theorem \ref{idea}, we obtain that $b\mid \ell F_{(n,5)}F_{aj}$ for every $j\geq 1$ where $\ell = 2,6$. So we need to consider only $v_2$, $v_3$, and $v_5$ of $b$ and $F_{aj}$. It is easy to check using Lemmas \ref{p-adic} and \ref{lcmN} that 
\begin{align*}
&\text{when $5\mid n$, $v_5(b)\leq v_5(F_{aj})$ if and only if $v_5(j) \geq 1$, and}\\
&\text{when $5\nmid n$, $v_5(b) \leq v_5(F_{aj})$ for every $j\geq 1$.}
\end{align*}
In addition, $v_2$ and $v_3$ of $b$ and $F_{aj}$ are
$$
v_2(b) = \begin{cases}
4, &\text{if $n\equiv 1,2,3,4,5,6\pmod{12}$};\\
v_2(n+12-r)+3,&\text{if $n\equiv r\pmod{12}$ and $7\leq r\leq 12$},
\end{cases}
$$
$$
v_2(F_{aj}) = \begin{cases}
v_2(n+4-r)+v_2(j)+2, &\text{if $n\equiv r\pmod{4}$ and $1\leq r\leq 2$};\\
v_2(n+4-r)+v_2(n+8-r)+v_2(j),&\text{if $n\equiv r\pmod{4}$ and $3\leq r\leq 4$},
\end{cases}
$$
$$
v_3(b) = \begin{cases}
1, &\text{if $n\equiv 1,2,5,6\pmod{12}$};\\
2, &\text{if $n\equiv 3,4\pmod{12}$};\\
v_3(n+12-r)+1,&\text{if $n\equiv r\pmod{12}$ and $r\in\{9,10\}$};\\
v_3(n+12-r)+2,&\text{if $n\equiv r\pmod{12}$ and $r\in\{7,8,11,12\}$},
\end{cases}
$$
$$
v_3(F_{aj}) = v_3(n+3-r)+v_3(n+6-r)+v_3(j), \text{if $n\equiv r\pmod{3}$ and $1\leq r\leq 3$}.
$$
\textbf{Case 1: $n\equiv 1 \pmod {4}$}. Then $b\mid 2F_{(n,5)}F_{aj}$ for every $j\geq 1$ and we only need to consider $v_p(b)$ and $v_p(F_{aj})$ for $p = 2,5$. If $n\equiv 1 \pmod {3}$, then $v_2(F_{aj})\geq v_2(b)$. So if $5\nmid n$, we can choose $j=1$ and obtain $z(b)=a$, and if $5\mid n$, we can choose $j=5$ and obtain $z(b)=5a$. Therefore $z(b) = (5,n)a$. If $n\equiv 2 \pmod {3}$, then $v_2(F_{aj}) \geq v_2(b)$ and we similarly obtain that $z(b)=(5,n)a$. If $n\equiv 0 \pmod {3}$, then $v_2(F_{aj})\geq v_2(b) \Leftrightarrow v_2(j)\geq 1$. Thus if $5\nmid n$, we can choose $j=2$ and obtain $z(b)=2a$, and if $5\mid n$, we can choose $j=10$ and obtain $z(b)=10a$. Therefore $z(b) = 2(5,n)a$.\\
\textbf{Case 2: $n\equiv 2 \pmod {4}$}. This case is similar to Case 1 and we obtain 
$$
z(b) = 
\begin{cases}
(5,n)a, &\text{if $n\equiv 0,2\pmod 3$};\\
2(5,n)a, &\text{if $n\equiv 1\pmod 3$}.
\end{cases}
$$
\textbf{Case 3: $n\equiv 3 \pmod {4}$}.  Then $b\mid 6F_{(n,5)}F_{aj}$ for every $j\geq 1$, and we need to consider $v_p(b)$ and $v_p(F_{aj})$ for $p = 2, 3, 5$.\\
\textbf{Case 3.1: $n\equiv 1 \pmod {3}$}. Then  
$$
v_2(b)\leq v_2(F_{aj}) \Leftrightarrow v_2(n+1)+v_2(j)\geq 3, \ \text{and} \ v_3(b)\leq v_3(F_{aj}) \Leftrightarrow v_3(n+2)+v_3(j)\geq 2.
$$ 
Note that $n+1 \equiv 0,4 \pmod 8$ and $n+2 \equiv 0,3,6 \pmod 9$.
\begin{itemize}
\item[(i)] If $n+1 \equiv 0 \pmod 8$ and $n+2 \equiv 0 \pmod 9$, then $v_2(j)=v_3(j)=0$, and so $z(b) = (5,n)a = \frac{72(5,n)a}{(8,n+1)(9,n+2)}$.
\item[(ii)] If $n+1 \equiv 0 \pmod 8$ and $n+2 \equiv 3,6 \pmod 9$, then $v_2(j)=0$ and $v_3(j)=1$, and so $z(b) = 3(5,n)a = \frac{72(5,n)a}{(8,n+1)(9,n+2)}$.
\item[(iii)] If $n+1 \equiv 4 \pmod 8$ and $n+2 \equiv 0 \pmod 9$, then $v_2(j)=1$ and $v_3(j)=0$, and so $z(b) = 2(5,n)a = \frac{72(5,n)a}{(8,n+1)(9,n+2)}$.
\item[(iv)] If $n+1 \equiv 4 \pmod 8$ and $n+2 \equiv 3,6 \pmod 9$, then $v_2(j)=v_3(j)=1$, and so $z(b) = 6(5,n)a = \frac{72(5,n)a}{(8,n+1)(9,n+2)}$.
\end{itemize}
\textbf{Case 3.2: $n\equiv 2 \pmod {3}$}. This case is similar to Case 3.1 and we obtain
\begin{align*}
v_2(b)&\leq v_2(F_{aj}) \Leftrightarrow v_2(n+5)+v_2(j)\geq 3, \\
v_3(b)&\leq v_3(F_{aj}) \Leftrightarrow v_3(n+4)+v_3(j)\geq 2, \; \text{and}\\
z(b) &= \frac{72(5,n)a}{(8,n+5)(9,n+4)}.
\end{align*}
\textbf{Case 3.3: $n\equiv 0 \pmod {3}$}. This case leads to $z(b) = (5,n)a$.\\
\textbf{Case 4: $n\equiv 0 \pmod {4}$}. Similar to Case 3, we obtain 
$$
z(b) = \begin{cases}
(5,n)a, &\text{if $n\equiv 1\pmod 3$};\\
\frac{72(5,n)a}{(8,n)(9,n+1)},&\text{if $n\equiv 2\pmod 3$};\\
\frac{72(5,n)a}{(8,n+4)(9,n+3)}, &\text{if $n\equiv 0\pmod 3$}.
\end{cases}
$$
This completes the proof.
\end{proof}
We can obtain the following result from the proof of Theorem \ref{order5}.
\begin{corollary}\label{cororder5}
Let $n\geq1$, $a=[n,n+1,\ldots,n+5]$, $b = F_nF_{n+1}\cdots F_{n+5}$, and $c = (5,n)$. Then
\[
z(b) =\begin{cases}
ac, & \text{if $n\equiv1,2,3,4,5,6 \pmod {12}$} \\
2ac, & \text{if $n\equiv9,10\pmod {12}$};\\
\frac{72(5,n)a}{(8,n+|r-8|)(9,n+|r-9|)}, &\text{if $n\equiv r\pmod{12}$ and $r\in \{7,8,12\}$};\\
\frac{72(5,n)a}{(8,n+5)(9,n+4)}, &\text{if $n\equiv 11\pmod{12}$}.
\end{cases}
\]
\end{corollary}
Next we give the formula of $z(F_nF_{n+1}\cdots F_{n+6})$. It is shorter to state it in the form similar to Corollary \ref{cororder5} than Theorem \ref{order5}. 
\begin{thm} \label{order6}
Let $n\geq1$, $a=[n,n+1,\ldots, n+6]$, $b = F_nF_{n+1}\cdots F_{n+6}$, and $c = (5,n(n+1))$. Then
$z(b) =$ 
\[
\begin{cases}
ac, & \text{if $n\equiv1,2,3,4,5 \pmod {12}$};\\
\frac{(64)(27)ac}{(64,n+2)(27,n(n+3))}, & \text{if $n\equiv 6 \pmod {24}$};\\
\frac{(8)(27)ac}{(27,n(n+3))}, & \text{if $n\equiv 18 \pmod {24}$};\\
\frac{72ac}{(8,n-r)(9,n-r)}, & \text{if $n\equiv r \pmod {12}$ and $r\in\{7,8\}$}; \\
4ac, & \text{if $n\equiv9 \pmod {12}$}; \\
\frac{72ac}{(8,n+6)(9,n+5)}, & \text{if $n\equiv 10 \pmod {12}$}; \\
\frac{72ac}{(8,n+5)(9,n+4)}, & \text{if $n\equiv 11 \pmod {12}$}; \\
\frac{(64)(27)ac}{(64,n+4)(27,(n+3)(n+6))}, & \text{if $n\equiv 0 \pmod {12}$}.
\end{cases}
\] 
\end{thm}
\begin{proof}
The proof of this theorem follows the same idea used previously. So we will only give the evaluation of $v_2$, $v_3$, and $v_5$ of $b$ and $F_{aj}$. Similar to the proof of Theorem \ref{order5}, we have 
\begin{align*}
&\text{when $5\mid n(n+1)$, $v_5(b)\leq v_5(F_{aj}) \Leftrightarrow v_5(j) \geq 1$,}\\
&\text{when $5\nmid n(n+1)$, $v_5(b)\leq v_5(F_{aj})$ for every $j \geq 1$,}
\end{align*}
$$
v_2(F_{aj}) = 
\begin{cases}
v_2(n+3)+v_2(j)+2, &\text{if $n\equiv 1\pmod 4$};\\
v_2(n+6)+v_2(j)+3, &\text{if $n\equiv 2\pmod 8$};\\
v_2(n+2)+v_2(j)+2, &\text{if $n\equiv 6\pmod 8$};\\
v_2(n+1)+v_2(n+5)+v_2(j), &\text{if $n\equiv 3\pmod 4$};\\
v_2(n)+v_2(n+4)+v_2(j), &\text{if $n\equiv 0\pmod 4$},
\end{cases}
$$
$$
v_3(F_{aj}) = 
\begin{cases}
v_3(n+2)+v_3(n+5)+v_3(j), &\text{if $n\equiv 1\pmod 3$};\\
v_3(n+1)+v_3(n+4)+v_3(j), &\text{if $n\equiv 2\pmod 3$};\\
v_3(n)+v_3(n+3)+v_3(n+6)+v_3(j)-1, &\text{if $n\equiv 0\pmod 3$},
\end{cases}
$$
$$
v_2(b) = 
\begin{cases}
4, &\text{if $n\equiv 1,2,4,5\pmod {12}$};\\
5, &\text{if $n\equiv 3\pmod {12}$};\\
v_2(n+12-r)+3, &\text{if $n\equiv r\pmod {12}$ and $r\in \{7,8,10,11\}$};\\
v_2(n+3)+4, &\text{if $n\equiv 9\pmod {12}$};\\
v_2(n+12-r)+6, &\text{if $n\equiv r\pmod {12}$ and $r\in \{6,12\}$},
\end{cases}
$$
$$
v_3(b) = 
\begin{cases}
1, &\text{if $n\equiv 1,5\pmod {12}$};\\
2, &\text{if $n\equiv 2,3,4\pmod {12}$};\\
v_3(n+12-r)+2, &\text{if $n\equiv r\pmod {12}$ and $r\in \{6,7,8,10,11,12\}$};\\
v_3(n+3)+1, &\text{if $n\equiv 9\pmod {12}$}.
\end{cases}
$$
\end{proof}
\section{The Case of Lucas Numbers}\label{sec4}

Recall that Marques \cite{Ma1} and Marques and Trojovsk\'y \cite{MTr} obtain, respectively, the formula for $z(L_nL_{n+1}\cdots L_{n+k})$ in the case $1\leq k\leq 3$ and in the case $k=4$. Our method can be applied to any case $k\geq 1$. We give the outline of the calculation as follows.

First of all, similar to Lemma \ref{lcmF}, we need a formula for the least common multiple of consecutive Lucas numbers, which is given below.
\begin{lemma} \label{lcmL} 
For each $k\geq 1$, let $P_k = L_nL_{n+1}L_{n+2}\cdots L_{n+k}$. Then the following statements hold for every $n \geq 1$.
\begin{itemize}
\item[(i)] $[L_n,L_{n+1}]=L_nL_{n+1}$. 
\item[(ii)] $[L_n,L_{n+1},L_{n+2}]=L_nL_{n+1}L_{n+2}$. 
\item[(iii)] $[L_n,L_{n+1},L_{n+2},L_{n+3}]=\frac{P_3}{F_{(n,3)}}$. 
\item[(iv)] $
[L_n,L_{n+1},L_{n+2},L_{n+3},L_{n+4}]=\begin{cases}
\frac{P_4}{F_{(n-2,4)}}, & \text{if $n\equiv1\pmod {3}$};\\
\frac{P_4}{2F_{(n-2,4)}}, & \text{if $n\equiv0,2\pmod {3}$}.
\end{cases}$ 
\item[(v)] $
[L_n,L_{n+1},L_{n+2},\ldots,L_{n+5}]=\begin{cases}
\frac{P_5}{6}, & \text{if $n\equiv1,2\pmod {4}$};\\
\frac{P_5}{2}, & \text{if $n\equiv0,3\pmod {4}$}.
\end{cases}$
\item[(vi)] 
$[L_n,L_{n+1},L_{n+2},\ldots,L_{n+6}]
= \begin{cases}
\frac{P_6}{3\cdot 2^{v_2(L_n)+1}}, & \text{if $n\equiv0,1,2\pmod {4}$};\\
\frac{P_6}{2^{v_2(L_n)+1}}, & \text{if $n\equiv3\pmod {4}$}.
\end{cases}$
\end{itemize}
\end{lemma}  
In the proof of Lemma \ref{lcmL} and the others, it is useful to recall the following well-known results.
\begin{lemma}\label{ABClucas}
Let $m$, $n$ be positive integers and $d = (m,n)$. Then the following statements hold.
\begin{itemize}
\item[(i)] For $n\geq 2$, $L_{n}\mid F_{m}$ if and only if $2n\mid m$.
\item[(ii)] $(L_m, L_n) = \begin{cases}
L_d, &\text{if $\frac md$ and $\frac nd$ are odd};\\
2, &\text{if ($\frac md$ or $\frac nd$ is even) and $3\mid d$};\\
1, &\text{if ($\frac md$ or $\frac nd$ is even) and $3\nmid d$}. 
\end{cases}$
\item[(iii)] $(F_m, L_n) = \begin{cases}
L_d, &\text{if $\frac md$ is even and $\frac nd$ is odd};\\
2, &\text{if ($\frac md$ is odd or $\frac nd$ is even) and $3\mid d$};\\
1, &\text{if ($\frac md$ is odd or $\frac nd$ is even) and $3\nmid d$}. 
\end{cases}$
\end{itemize}
\end{lemma}
Lengyel's result on $p$-adic orders of Lucas numbers is also an important tool.
\begin{lemma} \label{p-adicl} (Lengyel \cite{Le})
For each $n\geq1$, let $v_{p}(n)$ be the $p$-adic order of $n$. For all primes $p\neq 5$, we have
$$
v_{2}(L_{n})=\begin{cases}
0, & \text{if $n\equiv1,2\pmod {3}$};\\
2, & \text{if $n\equiv3\pmod {6}$};\\
1, & \text{if $n\equiv0\pmod {6}$, and}
\end{cases}
$$
$$
v_{p}(L_{n})=\begin{cases}
v_{p}(n)+v_{p}(F_{z(p)}), & \text{if $z(p)$ is even and $n \equiv \frac{z(p)}{2}\pmod {z(p)}$};\\
0, & \text{otherwise}.
\end{cases}
$$
\end{lemma}
Finally, we adjust Theorem \ref{idea} a little, so that it is easily applied in the Lucas case.
\begin{theorem} \label{ideal}
Let $n\geq 2$, $k\geq 1$, $a = 2[n,n+1,\ldots,n+k]$, $b = L_nL_{n+1}\cdots L_{n+k}$ and $f_k(n)=\frac{L_nL_{n+1}L_{n+2}\cdots L_{n+k}}{[L_n,L_{n+1},L_{n+2},\ldots ,L_{n+k}]}$. Then the following holds.
\begin{itemize}
\item[(i)] $b\mid f_k(n)F_{aj}$ for every $j\geq 1$.
\item[(ii)] $z(b)=aj$ where $j$ is the smallest positive integer such that $b\mid F_{aj}$. In fact, $j$ is the smallest positive integer such that $v_p(b) \leq v_p(F_{aj})$ for every prime $p$ dividing $f_k(n)$.
\end{itemize}
\end{theorem}
By Lemmas \ref{lcmL}, \ref{ABClucas}, \ref{p-adicl}, and Theorem \ref{ideal}, we can calculate $z(L_nL_{n+1}\cdots L_{n+k})$ for every $1\leq k\leq 6$. Below is the formula of $z(L_nL_{n+1}\cdots L_{n+k})$ when $k = 4, 5 ,6$.
\begin{theorem}\label{mainlucas}
Let $n\geq 1$, $a = 2[n,n+1,\ldots,n+k]$, and $b = L_nL_{n+1}\cdots L_{n+k}$. Then the following statements hold.
\begin{itemize}
\item[(i)] If $k = 4$, then $z(b) =
\begin{cases}
3a, & \text{if $n\equiv 2,14, 18, 30 \pmod {36}$}; \\
a, & \text{otherwise}.
\end{cases}$ 
\item[(ii)] If $k = 5$, then $z(b) =\begin{cases}
3a, & \text{if $n\equiv1,2,13,14,17,18,29,30 \pmod {36}$}; \\
a, & \text{otherwise}.
\end{cases}$
\item[(iii)] If $k = 6$, then $z(b) = \begin{cases}
3a, & \text{if $n\equiv1,2,12,13,14,16,17,18,28,29 \pmod {36}$};\\
a, & \text{otherwise}.
\end{cases}$
\end{itemize}
\end{theorem}
Remark that our formula for $z(L_nL_{n+1}L_{n+2}L_{n+3}L_{n+4})$ may look different from that of Marques and Trojovsk\'y \cite{MTr} but it is actually the same after substituting $a = 2[n,n+1,\ldots,n+4]$ using Lemma \ref{lcmN}. We would like this section to be informative but not too long, so we only give a proof of part (ii) of Theorem \ref{mainlucas}.

\begin{proof}[Proof of Part (ii) of Theorem \ref{mainlucas}.]
It is easy to check that the result holds for $n=1,2$. So assume that $n\geq 3$. By Lemma \ref{lcmL} and Theorem \ref{ideal}, we obtain that $b\mid \ell F_{aj}$ for every $j\geq 1$ where $\ell = 2,6$. So we need to consider only $v_2$ and $v_3$ of $b$ and $F_{aj}$. Remark that $4\mid aj$ and $3\mid aj$. So by Lemma \ref{p-adic}, we obtain $v_2(F_{aj}) = v_2(aj)+2\geq 4$. For $n\equiv 0\pmod 3$, we obtain by Lemma \ref{p-adicl} that $v_2(b) = v_2(L_n)+v_2(L_{n+3}) = 3$. Similarly, if $n\equiv 1\pmod 3$, then $v_2(b) = v_2(L_{n+2})+v_2(L_{n+5}) = 3$, and if $n\equiv 2\pmod3$, then $v_2(b) = v_2(L_{n+1})+v_2(L_{n+4}) = 3$. So in any case, 
\begin{equation}\label{eqstarthm3.5}
\text{$v_2(b) = 3 < v_2(F_{aj})$ for every $j\geq 1$.}
\end{equation}
 In addition,
\begin{itemize}
\item[(a)] if $n\equiv 0\pmod 4$, then $v_3(b) = v_3(L_{n+2}) = v_3(n+2)+1$,
\item[(b)] if $n\equiv 1\pmod 4$, then $v_3(b) = v_3(L_{n+1})+v_3(L_{n+5}) = v_3(n+1)+v_3(n+5)+2$, 
\item[(c)] if $n\equiv 2\pmod 4$, then $v_3(b) = v_3(L_n)+v_3(L_{n+4}) = v_3(n)+v_3(n+4)+2$,
\item[(d)] if $n\equiv 3\pmod 4$, then $v_3(b) = v_3(L_{n+3}) = v_3(n+3)+1$.
\end{itemize}
By Lemmas \ref{p-adic} and \ref{lcmN}, we obtain the following.
\begin{itemize}
\item[(e)] If $n\equiv 0\pmod 3$, then $v_3(F_{aj}) = v_3(aj)+1 = v_3(a)+v_3(j)+1$ $= v_3(n)+v_3(n+3)-1+v_3(j)+1 = v_3(n)+v_3(n+3)+v_3(j)$.
\item[(f)] If $n\equiv 1\pmod 3$, then $v_3(F_{aj}) = v_3(n+2)+v_3(n+5)+v_3(j)$.
\item[(g)] If $n\equiv 2\pmod 3$, then $v_3(F_{aj}) = v_3(n+1)+v_3(n+4)+v_3(j)$.
\end{itemize}
\textbf{Case 1:} $n\equiv 0,3\pmod 4$. Then by Theorem \ref{ideal}, Lemma \ref{lcmL}, and \eqref{eqstarthm3.5}, we can choose $j=1$ and obtain $z(b) = a$.\\
\textbf{Case 2:} $n\equiv 1,2\pmod 4$. Then by Theorem \ref{ideal}, Lemma \ref{lcmL}, and \eqref{eqstarthm3.5}, we only need to check $v_3$ of $b$ and $F_{aj}$.\\
\textbf{Case 2.1:} $n\equiv 1\pmod 4$ and $n\equiv 0\pmod 3$. Then by (b) and (e), we obtain 
\begin{align*}
v_3(b) = v_3(n+1)+v_3(n+5)+2 = 2 \leq v_3(n)+v_3(n+3)+v_3(j) = v_3(F_{aj})
\end{align*}
for every $j$. So we choose $j=1$ and obtain $z(b) = a$.\\
\textbf{Case 2.2:} $n\equiv 2\pmod 4$ and $n\equiv 0\pmod 3$. Then by (c) and (e), $v_3(b) = v_3(n)+2$ and $v_3(F_{aj}) = v_3(n)+v_3(n+3)+v_3(j)$. So $v_3(F_{aj})\geq v_3(b)$ if and only if $v_3(n+3)+v_3(j) \geq 2$. Therefore 
\begin{itemize}
\item[(i)] if $n+3\equiv 0\pmod 9$, then we choose $j=1$ and obtain $z(b) = a$,
\item[(ii)] if $n+3\equiv 3,6\pmod 9$, then we choose $j=3$ and obtain $z(b) = 3a$.
\end{itemize}
\textbf{Case 2.3:} $n\equiv 2\pmod 4$ and $n\equiv 1\pmod 3$. Similar to Case 2.1, we obtain $z(b) = a$.\\
\textbf{Case 2.4:} $n\equiv 1\pmod 4$ and $n\equiv 1\pmod 3$. This case is similar to Case 2.2 and we obtain that $v_3(F_{aj}) \geq v_3(b)$ if and only if $v_3(n+2)+v_3(j) \geq 2$. Therefore
\begin{itemize}
\item[(i)] if $n+2\equiv 0\pmod 9$, then $z(b) = a$,
\item[(ii)] if $n+2\equiv 3,6\pmod 9$, then $z(b) = 3a$.
\end{itemize} 
\textbf{Case 2.5:} $n\equiv 1\pmod 4$ and $n\equiv 2\pmod 3$. This case is similar to Cases 2.2 and 2.4 and we obtain that $v_3(F_{aj}) \geq v_3(b)$ if and only if $v_3(n+4)+v_3(j) \geq 2$. So 
\begin{itemize}
\item[(i)] if $n+4\equiv 0\pmod 9$, then $z(b) = a$,
\item[(ii)] if $n+4\equiv 3,6\pmod 9$, then $z(b) = 3a$.
\end{itemize} 
\textbf{Case 2.6:} $n\equiv 2\pmod 4$ and $n\equiv 2\pmod 3$. This case is similar to Cases 2.2, 2.4, and 2.5 and we obtain that $v_3(F_{aj}) \geq v_3(b)$ if and only if $v_3(n+1)+v_3(j) \geq 2$. So
\begin{itemize}
\item[(i)] if $n+1\equiv 0\pmod 9$, then $z(b) = a$,
\item[(ii)] if $n+1\equiv 3,6\pmod 9$, then $z(b) = 3a$.
\end{itemize}
Combining the result in each case, we obtain the desired formula.
\end{proof}

 
\textbf{Acknowledgments} The first author receives a scholarship from DPST of IPST Thailand. The second author currently receives financial support jointly from The Thailand Research Fund and Faculty of Science, Silpakorn University, grant number RSA5980040. Correspondence should be addressed to Prapanpong Pongsriiam: prapanpong@gmail.com.

\end{document}